\documentclass[a4paper,10pt]{article}

\usepackage[all]{xy}

\usepackage{amsmath}
\usepackage{amssymb}
\usepackage{amsthm}
\usepackage[normalem]{ulem}
\usepackage{tensor}  
\newcommand{\sba}[1][n]{\ensuremath{b^{\prime}_{#1}}}
\newcommand{\Field}{\ensuremath{\mathbb{F}}}
\newcommand{\Parameters}{\ensuremath{\delta, \delta_L, \delta_R, \kappa_L, \kappa_R, \kappa}}
\newcommand{\themodule}{\ensuremath{\tensor[_{\sba}]{\mathcal{V}(n)}{}}}

\DeclareMathOperator\algdim{dim}

\DeclareMathOperator\amod{mod}

\DeclareMathOperator\coker{Coker}
\DeclareMathOperator\End{End}

\DeclareMathOperator\Hom{Hom}
\DeclareMathOperator\op{op}

\DeclareMathOperator\res{res}
\DeclareMathOperator\TL{TL}

\newcommand{\iden}{{\rm 1\hspace*{-0.4ex}%
\rule{0.1ex}{1.52ex}\hspace*{0.2ex}}}

\theoremstyle{plain} \newtheorem{theorem}{Theorem}[section]
                     \newtheorem{proposition}[theorem]{Proposition}
                     \newtheorem{lemma}[theorem]{Lemma}
                     \newtheorem{corollary}[theorem]{Corollary}
\theoremstyle{definition} \newtheorem{define}[theorem]{Definition}
                          \newtheorem{remark}[theorem]{Remark}
                          \newtheorem{example}[theorem]{Example}
                                                    
\title{Tilting Modules for the \\ Symplectic Blob Algebra}
\author{Andrew Reeves}

\begin{document}

\maketitle

\begin{abstract}
Let $\Field$ be an algebraically closed field.  For $n \in \mathbb{N}$ and $\delta, \delta_L, \delta_R, \kappa_L, \kappa_R, \kappa \in \Field$, the symplectic blob algebra $\sba(\delta, \delta_L, \delta_R, \kappa_L, \kappa_R, \kappa)$ is a finite dimensional non-commutative $\Field$-algebra that may be viewed as an extension of the Temperley-Lieb algebra.

In a previous paper, we defined, for any $n \in \mathbb{N}$, a tensor space module $\tensor[_\sba]{\mathcal{V}(n)}{}$.  In this paper we generalise an argument used by Martin and Ryom-Hansen in their study of the (ordinary) blob algebra to show that when $\sba$ is quasihereditary the module $\tensor[_\sba]{\mathcal{V}(n)}{}$ is full-tilting.
\end{abstract}

\section{Introduction}

The symplectic blob algebras are a family of finite dimensional associative algebras introduced in \cite{greenmartinparker2007}.  The study of their representation theory has applications to statistical mechanics \cite{nichols2006} and to the representation theory of the type $\widetilde{C}$ Hecke algebra \cite{degiernichols2009}.

Fix an algebraically closed field $\Field$.  Let $n \in \mathbb{N}$.  Then for $\Parameters \in \Field$, the symplectic blob algebra $\sba = \sba(\Parameters)$ is a finite dimensional $\Field$-algebra with generators $\{e,U_1.U_2. \ldots, U_{n-1}, f\}$ satisfying certain relations (which are given in Definition \ref{define:sbarelations}).  In particular $\sba$ is an extension of the Temperley-Lieb algebra $\TL_n(\delta)$ and of the blob algebra $b_n(\delta,\delta_L,\kappa_L)$.  These algebras are isomorphic to the subalgebras of $\sba$ generated by $\{U_1, U_2, \ldots, U_{n-1}\}$ and by $\{ e, U_1, U_2, \ldots, U_{n-1} \}$ respectively.

If $\Parameters \in \Field^{\times} = \Field \setminus \{ 0 \}$ then the algebra $\sba(\Parameters)$ is quasihereditary.   This means that for each simple module $L$ there is a corresponding standard module $\Delta$ and a corresponding (indecomposable) tilting module $T$ \cite[Appendix C]{dengduparshallwang} \cite[Appendix]{donkin}. 

In \cite{greenmartinparker2007} the authors find a poset $\Lambda_n = \{-n, -(n-1), \ldots, (n-1)\}$ which labels the simple  modules.  As $\sba$ is quasihereditary this set also labels the standard modules and the indecomposable tilting modules.  They give bases for each of the standard modules $\{\Delta_n(\lambda)) \; : \; \lambda \in \Lambda_n \}$.  

Even with this type of information it is often a challenging exercise to explicitly construct tilting modules for a given quasihereditary algebra.  Indeed, even the simple content of the standard modules $\{ \Delta_n(\lambda) \}$ is not known in the non-generic case \cite{greenmartinparker2008}.  This paper considers the closely related problem of constructing \emph{full} tilting modules for the symplectic blob algebra.  (Such a module is a direct sum of the indecomposable tilting modules, with each of these modules appearing as a summand at least once.)  In light of the previous remarks, it might seem difficult to attempt to construct a full-tilting $\sba$-module even for a particular $n \in \mathbb{N}$.  

However, it is known that for both the Temperley-Lieb and blob algebras, it is possible to construct such a full-tilting module for \emph{every} $n \in \mathbb{N}$ \cite{martinryomhansen2004}.  The connections between $\sba$ and these two algebras suggest that a similar construction may work for these algebras as well. In fact, this proves to be the case -- we show that the module $\tensor[_{\sba}]{\mathcal{V}(n)}{}$ constructed in \cite{reeves2011} is a full tilting module for every $n \in \mathbb{N}$.  

In order to do this it is necessary to first prove some preliminary results.  

In section 3 we recall (for example, from \cite{greenmartinparker2008}) the existence of two \emph{localisation functors} $F_e$ and $F_f$ associated with the generators $e$ and $f$.  

\begin{displaymath}
F_e, F_f : \sba-\amod \rightarrow \sba[n-1]-\amod \rm{.}
\end{displaymath}

In Proposition \ref{prop:Viso} of Section 4 we show that $$F_e \left( \tensor[_{\sba}]{\mathcal{V}(n)}{} \right) \cong F_f \left( \tensor[_{\sba}]{\mathcal{V}(n)}{} \right) \cong \tensor[_{\sba[n-1]}]{\mathcal{V}(n-1)}{}$$.  

The left adjoint of $F_e$ is a globalisation functor $G_e$.  In general this functor is not also the right adjoint of $F_e$; the modules $\tensor[_{\sba}]{\mathcal{V}(n)}{}$ and $(G_e \circ F_e) \tensor[_{\sba}]{\mathcal{V}(n)}{}$ are not isomorphic.. In Proposition \ref{prop:adjointinjective} we show that $\tensor[_{\sba}]{\mathcal{V}(n)}{}$ is tilting for all $n \in \mathbb{N}$ provided that the ``adjointness map'' $\psi_n : (G_e \circ F_e) \tensor[_{\sba}]{\mathcal{V}(n)}{} \rightarrow \tensor[_{\sba}]{\mathcal{V}(n)}{}$ is injective for all $n \in \mathbb{N}$.

Having established these results, in Theorem \ref{thm:Vistilting} we show that the module $\tensor[_{\sba}]{\mathcal{V}(n)}{}$ is tilting.  This is done by proving that the map $\psi_n$ is injective for every $n \in \mathbb{N}$.  This requires some new combinatorial sequences and results.

Some additional work is needed to show that the module is actually full-tilting.  In Proposition \ref{prop:permutationmodules} it is noted that the module $\tensor[_{\sba}]{\mathcal{V}(n)}{}$ has a manifest decomposition into ``permutation modules'' $M_n(r)$.  In Lemmas \ref{lem:FonMnr} and \ref{lem:smallestpermutationistilting} it is noted that $F_e$ and $F_f$ respect this decomposition, and that $M_n(2n) \cong L_n(-n) \cong T_n(-n)$.  This leads to the proof of Theorem \ref{thm:fulltilt}: $\tensor[_{\sba}]{\mathcal{V}(n)}{}$ is a full tilting module.

\subsection*{Notation}

Fix the following notation for the rest of the paper.

Let $\Field$ be an algebraically closed field.  Let $\Field^{\times} = \Field \setminus \{ 0 \}$ be the units of this field.

For $q \in \Field^{\times}$ and $n \in \mathbb{N} \setminus {0}$ let
\begin{displaymath}
 [n]_q = \frac{q^{n} - q^{-n}}{q - q^{-1}} \rm{.}
\end{displaymath}

In particular, $[2]_q = q + \frac{1}{q}$.

Let $\uline{\Pi} = \left( \Parameters \right) \in \Field^{6}$.  Then let $\uline{\Pi_L} = (\delta, \kappa_L, \delta_R, \delta_L, \kappa_R, \kappa) $ and let $\uline{\Pi_R} = (\delta, \delta_L, \kappa_R, \kappa_L, \delta_R, \kappa)$.

Let $\uline{\Sigma} = \left( a,b,c,d,x,y,x,z \right) \in \Field^{8}$.  Then let $\uline{\Sigma_L} = (a,b,c,d,\frac{ab}{x},\frac{cd}{y},z,w)$ and let $\uline{\Sigma_R} = (a,b,c,d,x,y,\frac{bc}{z}, \frac{ad}{w})$.

The symplectic blob $\sba = \sba(\uline{\Pi})$ is the $\Field$-algebra described in Definition \ref{define:sbarelations} with generators $\mathcal{G}_n = \{ e, U_1, U_2, \ldots, U_{n-1}, f \}$.

The poset $\Lambda_n = \{ -n, -(n-1), \ldots, (n-2), (n-1) \}$ labels the simple modules $\{L_n(\lambda)\}$ of $\sba$, the standard modules $\{\Delta_n()\}$ of $\sba$ and the indecomposable tilting modules $\{T_n(\lambda)\}$ of $\sba$.

The tensor space module $\tensor[_{\sba}]{\mathcal{V}(n)}{}$ is recalled in Proposition \ref{thm:rep}.  It is shown in Proposition \ref{prop:permutationmodules} that it has a manifest decomposition into a direct sum of ``permutation modules'' $\{M_n(r)\}_{-2n \leq r \leq 2n}$ as defined in Definition \ref{def:permutationmodule}.

If $u \in \{1,2\}^{4n}$ then $\uline{u}$ is a corresponding vector in $\tensor[_{\sba}]{\mathcal{V}(n)}{}$. 

For any $q \in \Field^{\times}$, let 
\begin{displaymath}
 \uline{\widetilde{12}^{q}} = q \uline{12} + \uline{21} \rm{.}
\end{displaymath}

If no ambiguity results, let 
\begin{displaymath}
\uline{x} \otimes \uline{\widetilde{12}^{q}} \otimes \uline{y} = \uline{x \widetilde{12}^{q} y} \rm{.}
\end{displaymath}

\section{The Symplectic Blob Algebra $b^{\prime}_{n}$}

Let $\Field$ be an algebraically closed field.  Let $\Parameters \in \Field$.  Let $n \in \mathbb{N}$.

The following presentation of the symplectic blob algebra $\sba(\Parameters)$ over $\Field$ was given in \cite{greenmartinparker2008}.  The symplectic blob algebra was originally defined (in \cite{greenmartinparker2007}) in terms of certain planar diagrams; this presentation will not be used in this paper.

\begin{define}  \label{define:sbarelations}
Fix $n \in \mathbb{N}$.   Let $\Parameters \in \Field$. The symplectic blob algebra $\sba(\Parameters)$ is the unital, associative $\Field$-algebra with generators $\mathcal{G}_n = \{e, U_1, U_2, \ldots, U_{n-1}, f\}$ satisfying the relations below.

\begin{align}
 U_i^2  &= \delta U_i \qquad &\text{for all } i \label{rel1} \\
 U_i U_j U_i &= U_i \qquad &\text{if } |i-j| = 1 \label{rel9} \\ 
 U_i U_j &= U_j U_I \qquad &\text{if } |i-j| \neq 1 \label{rel8} \\ 
 e^2 &= \delta_L e \label{rel2} \\
 f^2 &= \delta_R f \label{rel3} \\
 U_1 e U_1 &= \kappa_L U_1 \label{rel4} \\
 U_{n-1} f U_{n-1} &= \kappa_R U_{n-1} \label{rel5} \\
 e U_i &= U_i e \qquad &\text{if } i \neq 1 \label{rel10} \\
 f U_i &= U_i f \qquad &\text{if } i \neq n-1 \label{rel11} \\
 ef &= fe \qquad &\text{if } n > 1 \label{rel12} \\
 IJI &= \kappa I \label{rel6} \\
 JIJ &= \kappa J \label{rel7}
\end{align}
where
\begin{equation*}
 I = \left\{ \begin{array}{lr}
              U_1 U_3 \ldots U_{n-2} f & \text{if } n \text{ is odd} \\
              U_1 U_3 \ldots U_{n-1} & \text{if } n \text{ is even}
             \end{array} \right.
\end{equation*}
and
\begin{equation*}
 J = \left\{ \begin{array}{lr}
              e U_2 \ldots U_{n-1} & \text{if } n \text{ is odd} \\
              e U_2 \ldots U_{n-2} f & \text{if } n \text{ is even}
             \end{array} \right. \text{.}
\end{equation*}
\end{define}

Note from the relations that $\sba \cong (\sba)^{\op}$.

It will be useful to define some particular elements in $\sba$.

\begin{define}
Fix $n \in \mathbb{N}$.   Let $\Parameters \in \Field^{\times}$.  

For $i \in \{-n, -(n-1), \ldots, (n-1) \}$, define $g_i \in \sba(\Parameters)$ as follows:
\begin{align}
g_{-n} &= 1 \\
g_{-(n-1)} = f \quad 
g_{(n-1)} &= e \quad
g_{(n-2)} = ef \\
g_{-(m-2)} &= U_{2m-1} g_{-m} \qquad \mbox{for } 3 \leq m \leq n \\
g_{(m-2)} &= U_{2(m-1)} g_{m} \qquad \mbox{for } 3 \leq m < n \\
g_0 &= I \rm{.}
\end{align}
\end{define}

\begin{proposition}  \label{prop:sbafacts}
Fix $n \in \mathbb{N}$.   Let $\Parameters \in \Field^{\times}$.  Then $\sba(\Parameters)$ is quasihereditary, with hereditary chain:
\begin{equation}
0 \subset \sba g_0 \sba \subset \sba g_{-1} \sba \subset \sba g_{1} \sba \subset \sba g_{-2} \sba \ldots \subset \sba g_{-n} \sba = \sba \rm{.}
\end{equation}
The standard modules $\{\Delta_n(\lambda)\}$ are indexed by the poset
\begin{displaymath}
\Lambda_n = \{ -n, -(n-1), -(n-2), \ldots, (n-2), (n-1) \}  \rm{,}
\end{displaymath}
with partial order illustrated by the Hasse diagram:
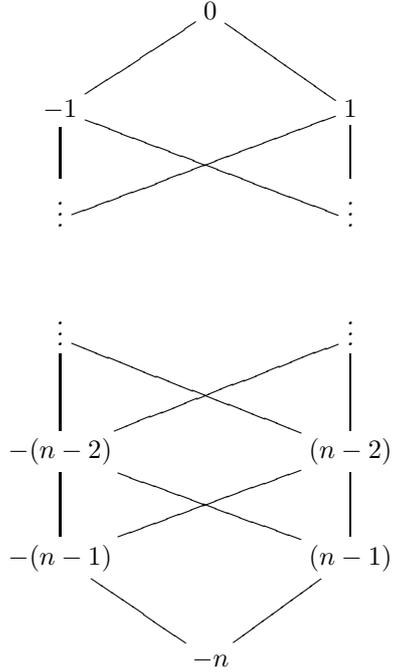
\begin{displaymath}
 \xymatrix{
& 0 \ar@{-}[dr] \ar@{-}[dl] & \\
-1 \ar@{-}[d] \ar@{-}[drr] && 1 \ar@{-}[d] \ar@{-}[dll] \\
\vdots && \vdots \\
\vdots \ar@{-}[d] \ar@{-}[drr] && \vdots \ar@{-}[d] \ar@{-}[dll] \\
-(n-2) \ar@{-}[d] \ar@{-}[drr] && (n-2) \ar@{-}[d] \ar@{-}[dll] \\
-(n-1) \ar@{-}[dr] && (n-1) \ar@{-}[dl] \\
& -n &
}  \rm{,}
\end{displaymath}
where $0$ is the maximal element and $-n$ is the minimal element.

For any $\lambda \in \Lambda_n$, the dimension of the standard module $\Delta_n(\lambda)$ over $\Field$ is given by
\begin{align}
\algdim_{\Field} \Delta_n(\lambda) &= \sum_{i=0}^{k(n,\lambda)} \begin{pmatrix} n \\ i \end{pmatrix} \rm{,}
\end{align}
where
\begin{displaymath}
k(n,\lambda) = \left\{ \begin{array}{cl}
                       n & \mbox{if } \lambda = 0 \\
                      \frac{n+\lambda}{2} & \mbox{if } \lambda < 0 \mbox{ and } n+\lambda \mbox{ is even} \\
                      \frac{n-\lambda-2}{2} & \mbox{if } \lambda > 0 \mbox{ and } n+\lambda \mbox{ is even} \\
                      \frac{n-|\lambda|-1}{2} & \mbox{otherwise}
                      \end{array} \right. \rm{.}
\end{displaymath}
\end{proposition}
\begin{proof}
See Section 8 of \cite{greenmartinparker2007}.
\end{proof}

\section{Homological Properties of the Symplectic Blob Algebra $b^{\prime}_{n}$}

This section defines functors between the module categories $\sba-\amod$ and $\sba[n-1]-\amod$ called \emph{localisation functors}.  Some basic properties of these functors are given.

Let $\Parameters \in \Field^{\times}$.  Let $\uline{\Pi} := (\Parameters)$.  Let $\uline{\Pi_L} := (\delta, \kappa_L, \delta_R, \delta_L, \kappa_R,\kappa)$.  Let $\uline{\Pi_R} := (\delta,\delta_L,\kappa_R,\kappa_L,\delta_R,\delta)$.   

Let $\sba = \sba(\uline{\Pi})$ be the symplectic blob algebra over $\Field$ as defined in Definition \ref{define:sbarelations}.  

Write $\mathcal{F}_n(\Delta)$ for the full subcategory of $\sba-\amod$ containing all modules which have a filtration by standard modules.  Write $\mathcal{F}_n(\nabla)$ for the full subcategory of $\sba-\amod$ containing all modules which have a filtration by costandard modules.

Define two idempotent elements $\bar{e}$ and $\bar{f}$ as follows.
\begin{align}
\bar{e} &= \frac{1}{\delta_L} e \\
\bar{f} &= \frac{1}{\delta_R} f \rm{.}
\end{align}

\begin{proposition}  \label{prop:twoisomorphsisms}
There is an isomorphism:
\begin{align}
\bar{e} \sba(\uline{\Pi}) \bar{e} &\cong \sba[n-1](\uline{\Pi_L}) \rm{,}  \label{eq:eisomorph}
\end{align}
in which
\begin{align}
e &\mapsto 1 \\
eU_1e &\mapsto e \\
eU_ie &\mapsto U_{i-1} \qquad \mbox{for } i \neq 1 \\
efe &\mapsto f \rm{.}
\end{align}
There is also an isomorphism
\begin{align}
\bar{f} \sba(\uline{\Pi}) \bar{f} &\cong \sba[n-1](\uline{\Pi_R}) \rm{,}
\end{align}
in which
\begin{align}
f &\mapsto 1 \\
fU_{n-1}f &\mapsto f \\
fU_{n-i}f &\mapsto U_{(n-1)-i} \qquad \mbox{for } i \neq 1 \\
fef &\mapsto e \rm{.}
\end{align}
\end{proposition}
\begin{proof}
See Proposition 8.1.1 of \cite{greenmartinparker2007}.
\end{proof}

The isomorphisms given in Proposition \ref{prop:twoisomorphsisms} mean that the idempotent embedding functors corresponding to the idempotents $\bar{e}$ and $\bar{f}$ are of particular interest.  These functors are called \emph{localisation functors}.

\begin{define}
Let $n \in \mathbb{N}$.  Let $\Parameters \in \Field$.

Write $\sba := \sba(\Parameters)$.  Write $\sba[n-1,L] := \sba[n-1](\delta, \kappa_L, \delta_R, \delta_L, \kappa_R, \kappa)$.  Write $\sba[n-1,R] := \sba[n-1](\delta, \delta_L, \kappa_R, \kappa_L, \delta_R, \kappa)$.

Define a functor
\begin{displaymath}
F_e^{n} : \sba-\amod \rightarrow \sba[n-1,L]-\amod
\end{displaymath}
and a functor
\begin{displaymath}
F_f^{n} : \sba-\amod \rightarrow \sba[n-1,R]-\amod
\end{displaymath}
by the formulas
\begin{align}
F_e^{n}(-) &= \bar{f} (-) \\
F_f^{n}(-) &= \bar{e} (-)
\end{align}.
\end{define}

Some basic facts about these functors are summarised below.

\begin{proposition}  \label{prop:Fexact}
The functors $F_e^{n}$ and $F_f^{n}$ are exact.

Identify $\Lambda_{n-1}$ with the obvious subset of $\Lambda_n$.  Then
\begin{align}
F_e^{n} L_n(\lambda) &= L_{n-1}(-\lambda) \qquad \mbox{if } -\lambda \in \Lambda_{n-1} \\
F_e^{n} \Delta_n(\lambda) &= \Delta_{n-1}(-\lambda) \qquad \mbox{if } -\lambda \in \Lambda_{n-1} \\
F_e^{n} \nabla_n(\lambda) &= \nabla_{n-1}(-\lambda) \qquad \mbox{if } -\lambda \in \Lambda_{n-1} \rm{,}
\end{align}
and
\begin{align}
F_f^{n} L_n(\lambda) &= L_{n-1}(\lambda) \qquad \mbox{if } \lambda \in \Lambda_{n-1} \\
F_f^{n} \Delta_n(\lambda) &= \Delta_{n-1}(\lambda) \qquad \mbox{if } \lambda \in \Lambda_{n-1} \\
F_f^{n} \nabla_n(\lambda) &= \nabla_{n-1}(\lambda) \qquad \mbox{if } \lambda \in \Lambda_{n-1} \rm{.}
\end{align}
On the other hand,
\begin{align}
F_e^{n} L_n(-n) = F_e^{n} \Delta_n(-n) &= F_e^{n} \nabla_(-n) = 0 \\
F_e^{n} L_n(-(n-1)) = F_e^{n} \Delta_n(-(n-1)) &= F_e^{n} \nabla_n(-(n-1)) = 0 \\
F_f^{n} L_n(-n) = F_f^{n} \Delta_n(-n) &= F_f^{n} \nabla_n(-n) = 0 \\
F_f^{n} L_n(n-1) = F_f^{n} \Delta_n(n-1) &= F_f^{n} \nabla_n(n-1) = 0 \rm{.}
\end{align}
\end{proposition}
\begin{proof}
See for example Section 5 of \cite{greenmartinparker2008}.
\end{proof}
Observe that it follows from the proposition that $F_e^{n}$ and $F_f^{n}$ each map $\mathcal{F}_n(\Delta)$ to $\mathcal{F}_{n-1}(\Delta)$ and $\mathcal{F}_n(\nabla)$ to $\mathcal{F}_{n-1}(\nabla)$.

\begin{proposition}
The left adjoint of $F_e^{n}$ is a functor
\begin{displaymath}
G_e^{n} : \sba[n-1]-\amod \rightarrow \sba-\amod \rm{,}
\end{displaymath}
given by the formula
\begin{displaymath}
G_e^{n}(-) : \sba e \otimes_{ \sba[n-1] } - \rm{.}
\end{displaymath}
The composition $F_e^{n} \circ G_e^{n} \cong  \iden_{\sba[n-1]-\amod}$.
\end{proposition}
\begin{proof}
See for example Theorem 8.6 of \cite{assemsimsonskowronski}.
\end{proof}

\begin{remark}
The functor $F_f$ has a similar left adjoint, but this is not needed to prove the result.
\end{remark}

The functor $G_e^{n}$ is not exact and need not map $\mathcal{F}_{n-1}(\nabla)$ to $\mathcal{F}_n(\nabla)$.  However it does map $\mathcal{F}_{n-1}(\Delta)$ to $\mathcal{F}_{n}(\Delta)$.  This is shown in the following result.

\begin{proposition}  \label{prop:ryomhansenmartin}
Let $M \in \mathcal{F}_{n-1} (\Delta)$.  Then $G_e^{n} M \in \mathcal{F}_{n} (\Delta)$.  Moreover, for
any $\lambda \in \Lambda_n$,
\begin{align}
(G_e^{n} M \; : \; \Delta_n(\lambda)) &= \left\{ \begin{array}{cl}
                  (M \; : \; \Delta_{n-1}(-\lambda)) & \mbox{if } -\lambda \in \Lambda_{n-1} \\
                  0 & \mbox{otherwise}
                                                 \end{array} \right. \rm{.}
\end{align}
\end{proposition}
\begin{proof}
This is essentially the result proved for the \emph{ordinary} blob algebra in Proposition 6 of \cite{martinryomhansen2004}.  The proof given in that paper carries over almost verbatim.\end{proof}

Recall that the \emph{counit} of an adjunction $G \dashv F$,
\begin{displaymath}
 \xymatrix{
\mathcal{C} \ar@/_1pc/[rr]_{F} & &  \mathcal{D} \ar@/_1pc/[ll]_{G}
} \rm{,}
\end{displaymath}
is the natural transformation $\epsilon_{-} : G F \rightarrow \iden_{\mathcal{C}}$ such that, given any 
$m : GFy \rightarrow x$ in $\mathcal{C}$,  there is an $m^{\prime} : y \rightarrow x$ in $\mathcal{C}$ making the diagram below commute: 
\begin{displaymath}
\xymatrix{
y \ar@{.>}[d]_{m^{\prime}} & & GF y \ar[r]^{m} \ar@{.>}[d]_{GFm^{\prime}} &  x \ar@{=}[d] \\
x & & GF x \ar[r]_{\epsilon_x} & x
} \rm{.}
\end{displaymath}
(See, for example, Theorem IV.1.2 of \cite{maclane}).

\begin{lemma}  \label{lemma:sbaeasrightmodule}
Let $n \in \mathbb{N}$.  Let $\sba = \sba(\uline{\Pi})$.  As a right $\bar{e} \sba \bar{e}$ module, $\sba \bar{e}$ is generated by the set $\mathcal{X}_n := \{ x_k \}_{0 \leq k \leq n}$ defined by
\begin{align*}
x_0 &= e \\
x_i &= U_i x_{i-1} \qquad \mbox{for } 1 \leq i < n \\
x_n &= f x_{n-1} \rm{.}
\end{align*}
\end{lemma}
\begin{proof}
This follows at once from the algebra relations given in Definition \ref{define:sbarelations}.
\end{proof}

\begin{proposition} 
Let $n \in \mathbb{N}$.  Let $\sba = \sba(\uline{\Pi})$.  Let $M \in \sba-\amod$ be a finitely generated module with spanning set $\{v_i\}_{i \in I}$.

Then $\left( G_e^{n} \circ F_e^{n} \right) M$ has spanning set $\{x \otimes v_i\}_{x \in \mathcal{X}_n, i \in I}$.  The action of the generators $\mathcal{G}_n = \{e, U_1, \ldots, U_{n-1}, f\}$ of $\sba$ on these elements is given by
\begin{align*}
e \left( x_0 \otimes v_i \right) &= \delta_L \left( x_0 \otimes v_i \right) \\
U_j \left( x_j \otimes v_i \right) &= \delta \left( x_j \otimes v_i \right) \\
f \left( x_n \otimes v_i \right) &= \delta_R \left( x_n \otimes v_i \right) \\
U_j \left( x_{j-1} \otimes v_i \right) &= \left( x_j \otimes v_i \right) \\
f \left( x_{n-1} \otimes v_i \right) &= \left( x_n \otimes v_i \right) \\
U_j \left( x_{j+1} \otimes v_i \right) &= \left( x_j \otimes v_i \right) \\
U_k \left( x_j \otimes v_i \right) &= \frac{1}{\delta_L^2} \left( x_j \otimes (e U_k e v_i) \right) \qquad \mbox{if } k > j+1 \\
f \left( x_j \otimes v_i \right) &= \frac{1}{\delta_L^2} \left( x_j \otimes (e f e v_i) \right) \qquad \mbox{if } j < n-1 \\
e \left( x_j \otimes v_i \right) &= \frac{1}{\delta_L} \left( x_0 \otimes ( e U_{j} U_{j-1} \ldots U_1 e v_i) \right) \qquad \mbox{if } j \geq 1 \\
U_k \left( x_j \otimes v_i \right) &= U_j U_{j-1} \ldots U_{k+2} \left( x_k \otimes v_i  \right)  \qquad \mbox{if } k < j-1
\end{align*}
\end{proposition}
\begin{proof}
This follows, after some calculation, from the relations given in Definition \ref{define:sbarelations} and the definition of $\{x_j\}$ given in Lemma \ref{lemma:sbaeasrightmodule}.
\end{proof}

\begin{proposition}  \label{prop:whatispsi}
Let $n \in \mathbb{N}$.  Let $\uline{\Pi} \in (\Field^{\times})^{6}$ and let $\sba = \sba(\uline{\Pi})$.  Let $M \in \sba-\amod$ be a finitely generated module with spanning set $\{v_i\}_{i \in I}$.  

The counit $\epsilon_M : (G_e^{n} \circ F_e^{n}) M \rightarrow M$ is the $\Field$-linear map defined by
\begin{displaymath}
\epsilon(x \otimes v_i) = xv_i \rm{.}
\end{displaymath}
\end{proposition}
\begin{proof}
Since $G_e^{n}$ is the left adoint of $F_e^{n}$, there is an isomorphism
\begin{displaymath}
\Phi: \Hom_{\sba[n-1]} (N, F_e^{n} M) \cong \Hom_{\sba} (G_e^{n} N, M) \rm{,}
\end{displaymath}
for any $M \in \sba-\amod$ and $N \in \sba[n-1]-\amod$.  

If $\theta : v \mapsto e T(v) \in \Hom_{\sba[n-1]} (N, F_e^{n} M)$ for $T(v)$ some $\Field$-linear combination of the spanning set $\{v_i\}_{i \in I}$, then $\Phi(\theta) : x \otimes v \mapsto xT(v)$.

Let $N := F_e^{n} M$.  Let $\iden_{N}$ be the identity function on $N$.  Using Theorem IV.1.2 of \cite{maclane} once more, note that the counit $\epsilon_M$ is given by
\begin{displaymath}
\epsilon_M = \Phi (\iden_N) \in \Hom_{\sba} (G_e^{n} N, M) = \Hom_{\sba} ((G_e^{n} \circ F_e^{n}) M, M) \rm{.}           
\end{displaymath}

If $M$ has spanning set $\{v_i\}_{i \in I}$ then, for any $v_i$,
\begin{displaymath}
\iden_N (ev_i) = ev_i \rm{,}
\end{displaymath}
therefore 
\begin{displaymath}
\epsilon_M(x \otimes v) = xv \rm{,}
\end{displaymath}
as claimed.
\end{proof}

\section{The Tensor Space Module $\themodule$}

\subsection{Definition of the Module $\mathcal{V}(n)$}

Let $a,b,c,d,x,y,z,w \in \Field^{\times}$.

In $\cite{reeves2011}$ a construction is presented for a ``tensor space'' symplectic blob algebra module over the ground ring $\mathbb{Z}[a_0^{\pm 1},b_0^{\pm 1},c_0^{\pm 1},d_0^{\pm 1},x_0^{\pm 1},y_0^{\pm 1},z_0^{\pm 1},w_0^{\pm 1}]$.  This module passes to a $\sba(\uline{\Pi_0})$ module $\mathcal{V}(n)$ provided that the parameters $\uline{\Pi_0} \in \Field^{6}$ satisfy conditions recalled in Theorem \ref{thm:rep}.

Let $V_m$ be the $m$-dimensional $\Field$-vector space with basis $\{v_1, v_2, \ldots, v_m \}$.  Then the $4N$-fold tensor product
\begin{displaymath}
V_m^{\otimes 4N} := \underbrace{ V_m \otimes_{\Field} V_m \otimes_{\Field} \ldots \otimes_{\Field} V_m  }_{4N \mbox{ copies}} \rm{,}
\end{displaymath}
has basis 
\begin{displaymath}
\left\{ v_{\alpha_{-2N+1}} \otimes v_{\alpha_{-2N+2}} \otimes \ldots \otimes v_{\alpha_{2N}} \; | \; \alpha_k \in \{1, 2, \ldots, m\} \right\} \rm{.}
\end{displaymath}

There is a bijection between this basis of $V_m^{\otimes 4N}$ and words of length $4N$ in the alphabet $\{1,2,\ldots, m\}$.  It is useful to introduce the following notation.

\begin{define}
Define a map $\uline{\;} : \{ 1,2 \}^{4N} \rightarrow V^{\otimes 4N}$ by
\begin{displaymath}
 \uline{\alpha_{-(2N-1)} \alpha_{-(2N-2)} \ldots \alpha_{2N}} := v_{\alpha_{-2N+1}} \otimes v_{\alpha_{-2N+2}} \otimes \ldots \otimes v_{\alpha_{2N}} \rm{.}
\end{displaymath}
\end{define}

Now consider the tensor space $V_2^{\otimes 4n}$ for $n \in \mathbb{N}$.  Let $I_{n} = \{-2n+1,-2n+2, \ldots, 2n\}$.

\begin{define}  \label{def:Rmatrices}
Let $n \in \mathbb{N}$.   Let $q \in \mathcal{F}^{\times}$.

Define a family of operators $\{ R^{q}_i \}_{i \in I_{n}} \subset \End_{\mathcal{A}}( V^{\otimes 4n})$ as follows:

Let $\alpha = \uline{\alpha_{-2n+1} \ldots \alpha_{2n}} \in V^{\otimes 4n}$ be a basis element.  Then for $i \in I_{n} \setminus \{2n\}$ we define $R^{q}_{i}$ by 
\begin{align*}
 R^{q}_i \alpha &= \delta'(\alpha_i,\alpha_{i+1}) \left( q^{2 - \alpha_i} \uline{\alpha_{-2n+1} \ldots 1 2 \ldots \alpha_{2n}} \right. \\
& \qquad \qquad \qquad \qquad \qquad + \left.  q^{1 - \alpha_i} \uline{\alpha_{-2n+1} \ldots 2 1 \ldots \alpha_{2n}} \right) \rm{,}
\end{align*}
and we define $R^{q}_{2n}$ by
\begin{align*} 
R^{q}_{2n} \alpha &= \delta`(\alpha_{2n},\alpha_{-2n+1}) \left( q^{2 - \alpha_{2n}} \uline{2 \alpha_{-2n+2} \ldots \alpha_{2n-1} 1} \right. \\
& \qquad \qquad \qquad \qquad \qquad + \left. q^{1 - \alpha_{2n}} \uline{ 1 \alpha_{-2n+2} \ldots \alpha_{2n-1} 2 } \right) \rm{.}
\end{align*}
\end{define}

\begin{theorem} \label{thm:rep}
Fix $n \in \mathbb{N}$.  Let $a,b,c,d,x,y,z,w \in \Field^{\times}$.  Let $\Parameters \in \Field$.  Let $\sba = \sba(\Parameters)$.

Let $\mathcal{G}_n = \{ e, U_1, \ldots, U_{n-1}, f \}$ be the generators of $\sba$ given in Definition \ref{define:sbarelations}.  Let $I_{n} = \{-2n+1,-2n+2\ldots,2n\}$.  Let $\{R^q_i\}_{i \in I_{n}} \subset \End_{\Field} (V^{\otimes4n})$ be the operators defined in Definition \ref{def:Rmatrices}.

Define a map $\mathcal{R} : \mathcal{G}_n \rightarrow \End_{\Field}(V^{\otimes 4n})$ by 
 \begin{align}
  \mathcal{R}(U_i) &= R^a_{-n-i} R^b_{-n+i} R^c_{n-i} R^d_{n+i} \label{tsr:eq1} \\
  \mathcal{R}(e) &= R^x_{-n} R^y_{n} \label{tsr:eq2} \\
  \mathcal{R}(f) &= R^z_0 R^w_{2n} \rm{ .} \label{tsr:eq3} 
 \end{align}
Then $\mathcal{R}$ extends to a unique representation of $\sba$, also called $\mathcal{R}$, if and only if
 \begin{align}
  \delta &= \left( a + \frac{1}{a} \right) \left( b + \frac{1}{b} \right) \left( c + \frac{1}{c} \right) \left( d + \frac{1}{d} \right) \label{thmcon1} \\
  \delta_L &= \left( x + \frac{1}{x} \right) \left( y + \frac{1}{y} \right) \label{thmcon2} \\
  \delta_R &= \left( z + \frac{1}{z} \right) \left( w + \frac{1}{w} \right) \label{thmcon3} \\
  \kappa_L &= \left( \frac{ab}{x} + \frac{x}{ab} \right) \left( \frac{cd}{y} + \frac{y}{cd} \right) \label{thmcon4} \\
  \kappa_R &= \left( \frac{ad}{w} + \frac{w}{ad} \right) \left( \frac{bc}{z} + \frac{z}{bc} \right) \label{thmcon5} \\
  \kappa &= \left\{ \begin{array}{lr}
                    \frac{xy}{zw} + 2 + \frac{zw}{xy} & \mbox{if } n \mbox{ is odd} \\
                    \frac{abcd}{xyzw} + 2 + \frac{xyzw}{abcd} & \mbox{if } n \mbox{ is even}
                   \end{array} \right. \rm{ .} \label{thmcon6}
 \end{align}
\end{theorem}
\begin{proof}
This is a special case of Theorem 2.5 of \cite{reeves2011}.
\end{proof}

\begin{corollary}
For any $\Parameters \in \Field$ there exist $a,b,c,d,x,y,z,w \in \Field^{\times}$ such that the conditions of Theorem \ref{thm:rep} are satisfied.
\end{corollary}

With this action, the tensor space $V_2^{\otimes 4n}$ becomes a (left) $\sba$-module.  Write $\themodule$ for this module.

\subsection{Examples and some More Notation}

\begin{example}  \label{example:V1}
The $\sba[1]$-module $\mathcal{V}(1)$ has basis
\begin{displaymath}
\left\{  \uline{1111}, \uline{1112}, \uline{1121}, \ldots, \uline{2221}, \uline{2222} \right\} \rm{.}
\end{displaymath}

The algebra $\sba[1]$ is generated by the elements $e$ and $f$.  Both these elements act like $0$ on ten of the basis vectors of $\mathcal{V}(1)$, namely those in the sets
\begin{displaymath}
\left\{  \uline{1111}  \right\} \rm{,} \\
\left\{  \uline{1112}, \uline{1121}, \uline{1211}, \uline{2111}  \right\} \rm{,} \\
\left\{  \uline{2221}, \uline{2212}, \uline{2122}, \uline{1222}  \right\} \rm{,} \\
\left\{  \uline{2222}  \right\} \rm{.}
\end{displaymath}

Consider the subspace of $\mathcal{V}(1)$ spanned by the remaining basis vectors, ordered lexicographically as
\begin{displaymath}
\left\{ \uline{1122}, \uline{1212}, \uline{1221}, \uline{2112}, \uline{2121}, \uline{2211} \right\} \rm{.}
\end{displaymath}
The element $e$ acts on this subspace as the matrix
\begin{displaymath}
\begin{pmatrix} 0 & & & & & \\
                & xy & x & y & 1 & \\
                & x & \frac{x}{y} & 1 & \frac{1}{y} &  \\
                & y & 1 & \frac{y}{x} & \frac{1}{x} & \\
                & 1 & \frac{1}{y} & \frac{1}{x} & \frac{1}{xy} & \\
                & & & & & 0 \end{pmatrix}  \rm{,}
\end{displaymath}
while the element $f$ acts on this subspace as the matrix
\begin{displaymath}
\begin{pmatrix} \frac{z}{w} & \frac{1}{w} & & & z & 1 \\
                \frac{1}{w} & \frac{1}{zw} & & & 1 & \frac{1}{z} \\
                & & 0 & & & \\
                & & & 0 & & \\
                z & 1 & & & wz & w \\
                1 & \frac{1}{z} & & & w & \frac{w}{z}
\end{pmatrix}  \rm{.}
\end{displaymath}

These two matrices are both projections onto a rank $1$ subspace.  An eigenvector to eigenvalue $1$ of $e$ is
\begin{displaymath}
 \frac{1}{[2]_x[2]_y} \left( x \uline{12} + \uline{21} \right) \otimes \left( y \uline{12} + \uline{21} \right)  \rm{.}
\end{displaymath}

An eigenvector to eigenvalue $1$ of $f$ is 
\begin{displaymath}
\frac{1}{[2]_z [2]_w} \left( \uline{1} \otimes \left( w \uline{12} + \uline{21} \right) \otimes \uline{2} + w \uline{2} \otimes \left( w \uline{12} + \uline{21} \right) \otimes \uline{1} \right) \rm{.}
\end{displaymath}
\end{example}

\begin{example} \label{example:V2}
The $\sba[2]$-module $\mathcal{V}(2)$ has basis
\begin{displaymath}
\left\{  \uline{11111111}, \uline{11111112}, \uline{11111121}, \ldots, \uline{22222221}, \uline{22222222} \right\} \rm{.}
\end{displaymath}

The element $e \in \mathcal{V}(2)$ acts on this (ordered) basis as the matrix
\begin{align*}
\mathcal{R}(e) &= R^x_{-2} R^y_{2} \rm{.}
\end{align*}

The element $f \in \mathcal{V}(2)$ acts on this (ordered) basis as the matrix
\begin{align*}
\mathcal{R}(f) &= R^z_{0} R^w_{4} \rm{.}
\end{align*}

Both these matrices are projections onto a rank $4$ subspace of $\mathcal{V}(2)$ (compare to Example \ref{example:V1}).

A direct calculation then shows that an eigenvector to eigenvalue $1$ of $\mathcal{R}(e)$ is of the form
\begin{align*}
\frac{1}{[2]_x[2]_y} \left( \uline{v_{-3}} \otimes \left( x \uline{12} + \uline{21} \right) \otimes \uline{v_{0} v_{1}} \otimes \left( y \uline{12} + \uline{21} \right) \otimes \uline{v_{4}} \right)  \rm{,}
\end{align*}
while an eigenvector to eigenvalue $1$ of $\mathcal{R}(f)$ is of the form
\begin{align*}
\frac{1}{[2]_z [2]_w} \left( \uline{1} \otimes \uline{v_{-2} v_{-1}} \otimes \left( w \uline{12} + \uline{21} \right) \otimes \uline{v_{2} v_{3}} \otimes \uline{2} \right. \\
  + \left. w \uline{2} \otimes \uline{v_{-2} v_{-1}} \otimes \left( w \uline{12} + \uline{21} \right) \otimes \uline{v_{2} v_{3}} \otimes \uline{1} \right) \rm{.}
\end{align*}
\end{example}

These examples motivate the following definition.

\begin{define} \label{def:tildenotation}
Let $q \in \Field^{\times}$.

Define $\uline{\widetilde{12}^{q}} \in V_2^{\otimes 2}$ by
\begin{displaymath}
\uline{\widetilde{12}^{q}} := q \uline{12} + \uline{21} \rm{.}
\end{displaymath}

If no ambiguity results, let 
\begin{displaymath}
\uline{x} \otimes \uline{\widetilde{12}^{q}} \otimes \uline{y} = \uline{x \widetilde{12}^{q} y} \rm{.}
\end{displaymath}

Here it should be understood that $\uline{x}$ and $\uline{y}$ may be either standard basis vectors or linear combinations of the form $\uline{\widetilde{12}^{q}}$.
\end{define}

To establish the use of this notation, consider again the eigenvectors calculated in the previous examples.

\begin{example}
With this notation established, the eigenvector of $e$ to eigenvalue $1$ given in Example \ref{example:V1} can be written as
\begin{displaymath}
\frac{1}{[2]_x [2]_y} \uline{\widetilde{12}^{x} \widetilde{12}^{y}} \rm{.}
\end{displaymath}

The eigenvector of $e$ to eigenvalue $1$ given in Example \ref{example:V2} can be written as
\begin{displaymath}
\frac{1}{[2]_x [2]_y} \uline{v_{-3} \widetilde{12}^{x} v_{0} v_{1} \widetilde{12}^{y} v_{4}} \rm{.}
\end{displaymath}
\end{example}

\begin{example}  \label{example:squarebracketnotation}
By further direct calculations, it can be seen that for $\uline{\widetilde{12}^{a} \widetilde{12}^{b}}  \in \mathcal{V}(1)$ and $x,y \in \Field^{\times}$,
\begin{align*}
R^{x}_{0} \uline{\widetilde{12}^{a} \widetilde{12}^{b}} &= ab \uline{1122} + \frac{ab}{x} \uline{1212} + x \uline{2121} + \uline{2211} \rm{,} \\
R^{y}_{2} \uline{\widetilde{12}^{a} \widetilde{12}^{b}} &= \uline{1122} + \frac{ab}{y} \uline{1212} + y \uline{2121} + ab \uline{2211} \rm{.}
\end{align*}
\end{example}

This example motivates some more notation.

\begin{define}
Define $\uline{[1122]^{q,p}}$ and $\uline{11]^{q,p}[22} \in V_2^{\otimes 4}$ by

\begin{align*}
\uline{[1122]^{q,p}} &= q \uline{1122} + \frac{q}{p} \uline{1212} + p \uline{2121} + \uline{2211} \rm{,} \\
\uline{22]^{q,p}[11} &= \uline{1122} + \frac{q}{p} \uline{1212} + p \uline{2121} + q \uline{2211} \rm{.}
\end{align*}

As in Definition \ref{def:tildenotation}, when no ambiguity results we may write
\begin{displaymath}
\uline{x [1122]^{q,p} y} = \uline{x} \otimes \uline{[1122]^{q,p}} \otimes \uline{y} \rm{,}
\end{displaymath}
and 
\begin{displaymath}
\uline{22]^{q,p} w [11} = \uline{11w22} + \frac{q}{p} \uline{12w12} + p \uline{21w21} + q \uline{22w11} \rm{.}
\end{displaymath}
\end{define}

\begin{example}
The eigenvector of $f$ to eigenvalue $1$ given in Example \ref{example:V1} can be written as
\begin{displaymath}
\frac{w}{[2]_z [2]_w} \uline{[1122]}^{\frac{z}{w},z} = \frac{z}{[2]_z [2]_w} \uline{11]^{\frac{w}{z},w}[22} \rm{.}
\end{displaymath}
\end{example}

It will later prove useful to have the following map.

\begin{define}  \label{def:umapsfromV}
Order $\{1,2\}^{4n}$ lexicographically, so that the sequence $11\ldots1$ is the smallest element and $22\ldots2$ is the greatest element.

Let $w \in \mathcal{V}(n)$.  Then 
\begin{displaymath}
w = \sum_{v \in \{1,2\}^{4n}} A_{w,v} \uline{v} \rm{,}
\end{displaymath}
for some $A_{w,v} \in \Field$.

Now define a map $u: \mathcal{V}(n) \setminus \{ 0 \} \rightarrow \{1,2\}^{4n}$ as follows:

For any non-zero $w \in \mathcal{V}(n)$, let
\begin{displaymath}
 u(w) = \min \{ v \in \{1,2\}^{4n} \; | \; A_{w,v} \neq 0 \} \rm{.}
\end{displaymath}
\end{define}

Note in particular that $u(\uline{v}) = v$ for any $v \in \{1,2\}^{4n}$ and that $u(\widetilde{\uline{12}}^q) = 12$.

\subsection{Permutation Modules}

As Example \ref{example:V1} suggests, the $\sba$-module $\mathcal{V}(n)$ has a natural decomposition into `permutation modules'.  These modules are defined as follows.

\begin{define}  \label{def:permutationmodule}
Let $n\in \mathbb{N}$.  Let $s \in \mathbb{Z}$ such that $-2n \leq s \leq 2n$.  Let $M_n(s)$ be the subspace of $\mathcal{V}(n)$ containing all vectors with $(2n+s)$ $\uline{1}$s and $(2n-s)$ $\uline{2}$s.
\end{define}

\begin{proposition}  \label{prop:permutationmodules}
Let $n,r \in \mathbb{N}$.  Then
\begin{align}
\sba \circ M_n(r) &\subset M_n(r) \rm{.}
\end{align}
That is, each $M_n(r)$ is a submodule of $\themodule$.

The module $\themodule$ has the following direct sum decomposition:
\begin{align}
\themodule &= \bigoplus_{-2n \leq s \leq 2n} M_n(s) \rm{.}
\end{align}
\end{proposition}
\begin{proof}
Let $v \in \{1,2\}^{4n}$.  Multiplying $\uline{v}$ by matrices of the form $R^{q}_i$ does not change the number of $\uline{1}$s or $\uline{2}$s.  This suffices to prove the first claim.  But every vector $v \in \themodule$ is an element of some permutation module $M_n(r)$.  So the second claim follows at once.
\end{proof}

\subsection{Localisation of $\themodule$}

In this section, write $\mathcal{V}(n;\uline{\Sigma})$ for the module $\themodule$ defined in Theorem \ref{thm:rep} with parameters $\uline{\Sigma} = (a,b,c,d,x,y,z,w)$.

\begin{proposition}  \label{prop:Viso}
Let $\mathcal{V}(n;\uline{\Sigma})$ be the $\sba(\uline{\Pi})$-module defined in Theorem \ref{thm:rep}.  

Let $\uline{\Sigma_L} := (a,b,c,d,\frac{ab}{x},\frac{cd}{y},z,w)$ and $\uline{\Sigma_R} = (a,b,c,d,x,y,\frac{bc}{z},\frac{ad}{w})$.

Then $\mathcal{V}(n-1;\uline{\Sigma_L})$ is a $\sba[n-1,L]$-module and $\mathcal{V}(n-1;\uline{\Sigma_R})$ is a $\sba[n-1,R]$-module.  Moreover
\begin{align}
F_e^{n} \mathcal{V}(n;\uline{\Sigma}) &\cong \mathcal{V}(n-1;\uline{\Sigma_L}) \\
F_f^{n} \mathcal{V}(n;\uline{\Sigma}) &\cong \mathcal{V}(n-1;\uline{\Sigma_R}) \rm{.}
\end{align}
\end{proposition}
\begin{proof}
It is easily checked that the parameters $\Sigma_L$ and $\Sigma_R$ satisfy the conditions of Theorem \ref{thm:rep} for $\uline{\Pi_L} = (\delta,\kappa_L,\delta_R,\delta_L,\kappa_R,\kappa)$ and $\uline{\Pi_R} = (\delta,\delta_L,\kappa_R,\kappa_L,\delta_R,\kappa)$ respectively.  So $\mathcal{V}(n-1;\uline{\Sigma_L})$ and $\mathcal{V}(n-1;\uline{\Sigma_R})$ are modules as claimed.

It will be sufficient to prove only the first of the claimed isomorphisms.  The proof of the second is very similar.

Note that $F_e^{n}$ acts on $\mathcal{V}(n;\uline{\Sigma})$ as multiplication by the idempotent
\begin{align}
\mathcal{R}^{n}(e) &= \frac{1}{[2]_x [2]_y} R^x_{-n} R^y_{n} \rm{.}
\end{align}
This matrix acts like the identity on all but four tensor factors of $\mathcal{V}(n;\uline{\Sigma})$.  These are the four factors labeled $(-n)$,$(-n+1)$,$(n)$ and $(n+1)$.  On these factors $\mathcal{R}^{n}(e)$ acts as a rank $1$ projection.  An eigenvector of $\mathcal{R}^{n}(e)$ to eigenvalue $1$ is
\begin{align}
\ldots \otimes \left( x \uline{12} + \uline{21} \right) \otimes \ldots \otimes \left( y \uline{12} + \uline{21} \right) \otimes \ldots \rm{.}
\end{align}

Define an $\Field$-linear map
\begin{displaymath}
\theta : \mathcal{V}(n-1;\uline{\Sigma_L}) \rightarrow F_e^{n} \mathcal{V}(n;\uline{\Sigma})	
\end{displaymath}
as follows: if $\uline{\alpha} =  \uline{\alpha_{-2(n-1)+1} \ldots \alpha_{2(n-1)}} \in \mathcal{V}(n-1;\uline{\Sigma_L})$, then
\begin{align}
\theta(\uline{\alpha}) = \uline{\alpha_{-2(n-1)+1} \ldots \alpha_{-(n-1)} \widetilde{12}^{x} \alpha_{-(n-2)} \ldots \alpha_{(n-1)} \widetilde{12}^{y} \alpha_{n} \ldots \alpha_{2(n-1)}} \rm{.}
\end{align}

This map $\theta$ is clearly a bijection.  It is therefore a module isomorphism iff it is a module homomorphism.  

Let $\eta : \bar{e} \sba(\uline{\Pi}) \bar{e} \cong \sba[n-1](\uline{\Pi_L})$ be the isomorphism \eqref{eq:eisomorph} given in Proposition \ref{prop:twoisomorphsisms}.

Then $\theta$ is a homomorphism if, for every basis vector $\uline{\alpha} \in \mathcal{V}(n-1;\uline{\Sigma_L})$ and every $g \in \mathcal{G}_{n-1} = \{ e, U_1, U_2, \ldots, U_{n-2}, f \}$,
\begin{align*}
\theta(g \uline{\alpha}) &= \eta(g) \theta (\uline{\alpha}) \rm{.}
\end{align*}

It is not difficult to check that this holds for $g \in \{e, U_2, U_3, \ldots, U_{n-2}, f\}$.  If $g = U_1$ then $\eta(g) = e U_1 e$.  A direct calculation shows that
\begin{align*}
U_1 \theta(\uline{\alpha}) &= Q \uline{\alpha_{-2(n-1)+1} \ldots \widetilde{12}^{a} \widetilde{12}^{b} \ldots \widetilde{12}^{c} \widetilde{12}^{d} \ldots \alpha_{2(n-1)}} \rm{,}
\end{align*}
for $Q$ a function of $\alpha, a, b, c, d, x, y$ given by 
\begin{displaymath}
Q = \delta^{\prime}(\alpha_{-(n-1)},\alpha_{-(n-2)}) \delta^{\prime}(\alpha_{(n-1)},\alpha_{n}) \left( \frac{x}{ab} \right)^{\alpha_{-(n-1)} -1} \left( \frac{y}{cd} \right)^{\alpha_{(n-1)} - 1} \rm{,}
\end{displaymath}
where, for any $x,y \in \{1,2\}$,
\begin{displaymath}
\delta^{\prime}(x,y) = \left\{ \begin{array}{cl}
                               0 & \mbox{if } x = y \\
                               1 & \mbox{if } x \neq y 
                               \end{array}  \right.  \rm{.}
\end{displaymath}

Therefore
\begin{align*}
\eta(U_1) \theta(\uline{\alpha}) &= e U_1 \theta(\uline{\alpha}) \\
  &= \delta^{\prime}(\alpha_{-(n-1)},\alpha_{-(n-2)}) \delta^{\prime}(\alpha_{(n-1)},\alpha_{n}) \left( \frac{x}{ab} \right)^{2-\alpha_{-(n-1)}} \left( \frac{y}{cd} \right)^{2 - \alpha_{n-1}} \Theta(\uline{\alpha}) \rm{,}
\end{align*}
where
\begin{align*}
\Theta(\uline{\alpha}) &= \uline{\alpha_{-2(n-1)+1}  \alpha_{-n}} \otimes \uline{\mathfrak{v}_L} \otimes \uline{ \alpha_{-(n-3)} \ldots \alpha_{(n-2)}} \otimes \uline{\mathfrak{v}_R} \otimes \uline{\alpha_{n+1} \ldots \alpha_{2(n-1)}} \rm{,} \\
\uline{\mathfrak{v}_L} &= \frac{ab}{x} \uline{1 \widetilde{12}^x 2} + \uline{2 \widetilde{12}^{x} 1} \rm{,} \\
\uline{\mathfrak{v}_R} &= \frac{cd}{y} \uline{1 \widetilde{12}^y 2} + \uline{2 \widetilde{12}^{y} 1} \rm{.}
\end{align*}

Conversely, $U_1 \in \sba{n-1}(\uline{\Pi_L})$ acts on $\uline{\alpha} \in \mathcal{V}(n-1;\uline{\Sigma_L})$ as
\begin{align*}
U_1 \uline{\alpha} &= \delta^{\prime}(\alpha_{-(n-1)},\alpha_{-(n-2)}) \delta^{\prime}(\alpha_{(n-1)},\alpha_{n}) \uline{\beta} \rm{,} \\
\uline{\beta} &= \uline{ \alpha_{-2n-1} \ldots \alpha_{-(n-2)} \widetilde{12}^{\frac{ab}{x}} \alpha_{-(n-3)} \ldots \alpha_{(n-2)} \widetilde{12}^{\frac{cd}{y}} \alpha_{(n+1)} \ldots \alpha_{2(n-1)}}
\end{align*}
and so it is easily checked that 
\begin{align*}
\theta(U_1 \uline{\alpha}) &= \delta^{\prime}(\alpha_{-(n-1)},\alpha_{-(n-2)}) \delta^{\prime}(\alpha_{(n-1)},\alpha_{n}) \left( \frac{x}{ab} \right)^{2-\alpha_{-(n-1)}} \left( \frac{y}{cd} \right)^{2 - \alpha_{n-1}} \Theta(\uline{\alpha}) \\
 &= \eta(U_1) \theta(\uline{\alpha}) \rm{.}
\end{align*}

But this means that $\theta$ is a module isomorphism, as required.
\end{proof}

\subsection{An `Adjointness Map'} %% don't like the name at all

\begin{proposition}  \label{prop:imageunderpsi}
The image $\psi_n \left( \left( G_e^{n} \circ F_e^{n} \right) \mathcal{V}(n) \right)$ is spanned by elements from the three sets below.
\begin{displaymath}
\left\{ \uline{ v_L \widetilde{12}^{x} v_M \widetilde{12}^{y} v_R } \; | \; v_L, v_R \in \{ 1, 2 \}^{n-1} \; v_M \in \{ 1, 2 \}^{2n-2} \right\}  \rm{,}
\end{displaymath}
\begin{displaymath}
\left\{ \uline{ v_1 \widetilde{12}^{a} v_2 \widetilde{12}^{b} v_3 \widetilde{12}^{c} v_4 \widetilde{12}^{d} v_5} \; | \; 2 \leq i \leq n \; v_1, v_5 \in \{ 1, 2 \}^{n-i}, v_2, v_4 \in \{ 1, 2 \}^{2(i-2)}, v_3 \in \{ 1,2 \}^{2(n-i)} \right\} \rm{,}
\end{displaymath}
\begin{displaymath}
\left\{ \uline{22]^{(ad,w)} v_{L} [1122]^{(bc,z)} v_{R} [11} \; | \; v_L, v_R \in \{ 1,2 \}^{2n-4} \right\} \rm{.}
\end{displaymath}
\end{proposition}
\begin{proof}
Let $v = \uline{\alpha_{-2n+1} \ldots \alpha_{2n}}$ be a standard basis element of $\mathcal{V}(n)$.

Let $g_0 = e$ and for $i \in \{1,2, \ldots, n-1\}$ let $g_i = U_i g_{i-1}$.  Let $g_n = f g_{n-1}$.

By Proposition \ref{prop:whatispsi}, $(G_e^{n} \circ F_e^{n}) \mathcal{V}(n)$ is spanned by elements of the form $g_k \otimes v$ for $0 \leq k \leq n$.  Therefore $\psi_n \left( (G_e^{n} \circ F_e^{n}) \mathcal{V}(n) \right)$ is spanned by elements of the form $g_k v$.

A simple calculation shows that $g_0v = ev$ is either $0$ or a scalar multiple of $\uline{ v_L \widetilde{12}^{x} v_M \widetilde{12}^{y} v_R }$, with $v_L, v_R \in \{ 1, 2 \}^{n-1}$ and $v_M \in \{ 1, 2 \}^{2n-2}$.

Further calculations show that, for $1 \leq k < n$, $g_k v = U_{n-1} (g_{k-1} v)$ is either $0$ or a scalar multiple of some vector $\uline{ v_1 \widetilde{12}^{a} v_2 \widetilde{12}^{b} v_3 \widetilde{12}^{c} v_4 \widetilde{12}^{d} v_5}$ of the required from.

Finally, $g_n v$ is either $0$ or of the form 
\begin{align*}
f (g_{n-1} v) &= f \left( \widetilde{12}^{a} v_2 \widetilde{12}^{b} \widetilde{12}^{c} v_4 \widetilde{12}^{d} \right) \\
 &= \uline{22]^{(ad,w)} v_{L} [1122]^{(bc,z)} v_{R} [11} \rm{,}
\end{align*}
by a direct calculation in the spirit of Example \ref{example:squarebracketnotation}.
\end{proof}

\begin{define}
Let $n \in \mathbb{N}$.  

Recall that $\epsilon_{-}$ is the counit of the adjunction $G_e^{n} \dashv F_e^{n}$.

Let $\mathcal{V}(n) = \themodule$ and let $\psi_n = \epsilon_{\mathcal{V}(n)}$, so that by Proposition \ref{prop:whatispsi}
\begin{align*}
\psi_n : \left( G_e^{n} \circ F_e^{n} \right) \mathcal{V}(n) &\rightarrow \mathcal{V}(n) \\
g_k \otimes v &\mapsto g_k v \rm{.}
\end{align*}
\end{define}

\begin{proposition}  \label{prop:adjointinjective}
Suppose that 
\begin{displaymath}
  \psi_n : \left( G_e^{n} \circ F_e^{n} \right) \mathcal{V}(n) \rightarrow \mathcal{V}(n)
\end{displaymath}
is injective, for all $n \in \mathbb{N}$.

Then the module $\tensor[_\sba]{\mathcal{V}(n)}{}$ is tilting for all $n \in \mathbb{N}$.
\end{proposition}
\begin{proof}
Compare with Proposition 4 of \cite{martinryomhansen2004}.

Note that $\mathcal{V}(0)$ is (trivially) a tilting $\sba[0]$-module, since $\sba[0] \cong \Field$ is semisimple.

Fix $n \in \mathbb{N}$.  Suppose that $\psi_n$ is injective.  Then there is a short exact sequence
\begin{align}
\xymatrix{
0 \ar[r] & \left( G_e^{n} \circ F_e^{n} \right) \mathcal{V}(n) \ar[r]^-{\psi_n} & \mathcal{V}(n) \ar[r] & \coker \psi_n \ar[r] & 0
} \rm{.}  \label{eq:sesforpsi}
\end{align}

Since $F_e \circ G_e \cong \iden_{\sba[n-1]}$ and $F_e$ is exact by Proposition \ref{prop:Fexact}, applying $F_e$ to the above exact sequence gives another exact sequence:
\begin{displaymath}
\xymatrix{
0 \ar[r] & F_e^{n} \mathcal{V}(n) \ar[r] & F_e^{n} \mathcal{V}(n) \ar[r] & F_e^{n} \left( \coker \psi_n \right) \ar[r] & 0
} \rm{.}
\end{displaymath}

So $F_e^{n} \left( \coker \psi_n \right) = 0$ and hence the only composition factors of $\coker \psi_n$ are $L_n(-n)$ and $L_n(n-1)$.  But $L_n(-n) \cong \Delta_n(-n)$ and $L_n(n-1) \cong \Delta_n(n-1)$, so in particular $\coker \psi_n \in \mathcal{F}_{n}(\Delta)$.

Now argue by induction on $n$.  Suppose that $\mathcal{V}(n-1)$ is tilting.

By Proposition \ref{prop:Viso}, $F_e \mathcal{V}(n) \cong \mathcal{V}(n-1)$.  By Proposition \ref{prop:ryomhansenmartin}, it follows that $\left( G_e^{n} \circ F_e^{n} \right) \mathcal{V}(n) \cong G_e^{n} \mathcal{V}(n-1) \in \mathcal{F}_n(\Delta)$.

Therefore the short exact sequence \eqref{eq:sesforpsi} becomes one in which all extremal terms are elements of $\mathcal{F}_n(\Delta)$.  By definition of $\mathcal{F}_n(\Delta)$ then, $\mathcal{V}(n)$ is also a standard module.

Moreover, the isomorphism $\sba \cong (\sba)^{\op}$ is fixed by $\mathcal{V}(n)$, since each matrix $R^q_i$ is self-adjoint.  So $\mathcal{V}(n)$ is contravariant self-dual.  The costandard modules $\{ \nabla_n(\lambda)\}$ are simply the duals of the standard modules $\{ \Delta_n(\lambda)^{\op} \}$ of $(\sba)^{\op}$.  So as $\mathcal{V}(n) \in \mathcal{F}_n(\Delta)$, it must also be the case that $\mathcal{V}(n) \in \mathcal{F}_n(\nabla)$.  Hence $\mathcal{V}(n)$ is a tilting module, as required.
\end{proof}

\section{The Module $\tensor[_\sba]{\mathcal{V}(n)}{}$ is Full Tilting}

Proposition \ref{prop:adjointinjective} reduces the problem of showing that the module $\themodule$ is tilting to the problem of proving that a certain linear map between two finite dimensional vector spaces is injective.  The aim is to show this for all $n \in \mathbb{N}$.  This can be done by proving some new combinatorial results.

First, consider the following example, which shows that $\psi_1 : \left( G_e^{1} \circ F_e^{1} \right) \mathcal{V}(1) \rightarrow \mathcal{V}(1)$ is injective.  It follows that $\mathcal{V}(1)$ is a tilting module.  This example also serves as a base case for the proof by induction of Theorem \ref{thm:Vistilting}.

\begin{example}  \label{example:Vagain}
Let $n=1$.  Recall Example \ref{example:V1}.  In this example it was shown that $F_e^{1} \mathcal{V}(1)$ is spanned by a single element:
\begin{displaymath}
\uline{\widetilde{12}^{x}\widetilde{12}^{y}} = \left( x \uline{12} + \uline{21} \right) \otimes \left( y \uline{12} + \uline{21} \right) \rm{.}
\end{displaymath}

By Lemma \ref{lemma:sbaeasrightmodule} $\sba[1] e$ is generated by the single element $e$.

Therefore
\begin{align}
\left( G_e^{1} \circ F_e^{1} \right) \mathcal{V}(1) &= \sba[1] e \otimes_{\Field} \bar{e} \mathcal{V}(1)
\end{align}
is spanned by the element $e \otimes \uline{\widetilde{12}^{x}\widetilde{12}^{y}}$.

By the same Proposition, $\psi_1 (e \otimes \uline{\widetilde{12}^{x}\widetilde{12}^{y}}) = \uline{\widetilde{12}^{x}\widetilde{12}^{y}}$.  Since this is non-zero, it follows that $\psi_1$ is an injective map.

Moreover, the (one dimensional) module $F_e^{1} \mathcal{V}(1) \cong \Delta_0(0)$.

It follows by Proposition \ref{prop:ryomhansenmartin} that
\begin{align}
\left( (G_e^{1} \circ F_e^{1}) \mathcal{V}(1) \; : \; \Delta_{1}(\lambda) \right) &= \left\{ \begin{array}{cl}
       1 & \mbox{if } \lambda = 0 \\
       0 & \mbox{otherwise} 
       \end{array} \right. \rm{.}
\end{align}

Finally, since $\algdim_{\Field} \mathcal{V}(1) = 16$ and $\algdim_{\Field} \Delta_1(0) = 2$ it follows that $$\left( \mathcal{V}(1) \; : \; \Delta_{1}(-1)\right) = 14 \rm{.}$$
\end{example}

The following integer sequence is needed for the statement of Theorem \ref{thm:Vistilting}.

\begin{define}  \label{def:vsequence}
Define an integer sequence $\{v(n)\}_{n \in \mathbb{N}}$ by
\begin{align}
v(0) &= 1 \\
v(1) &= 14 \\
v(2) &= 224 \\
v(m+1) + v(m-1) &= 16 v(m)  \qquad \mbox{for } m \geq 2 \rm{.}
\end{align}
\end{define}

It is now possible to state
\begin{theorem}  \label{thm:Vistilting}
For all $n \in \mathbb{N}$:
\begin{itemize}
 \item[(i)] The map $\psi_n : G_e^{n} \circ F_e^{n} \left( \mathcal{V}(n) \right) \rightarrow \mathcal{V}(n)$ is injective, and so $\mathcal{V}(n)$ is tilting.
 \item[(ii)] For all $\lambda \in \Lambda_n$ we have $\left[ \mathcal{V}(n) : \Delta(\lambda) \right] = v(|\lambda|)$.
 \item[(iii)] If $d_n$ is the dimension of $G_e^{n} \circ F_e^{n} \left( \mathcal{V}(n) \right)$, then $d_n = 16^{n} - v(n-1) - v(n)$.
 \item[(iv)] The image $\psi_n \left(G_e^{n} \circ F_e^{n} \left( \mathcal{V}(n) \right) \right)$ has basis $\mathcal{E}_n := \mathcal{E}_n^{A} \cup \mathcal{E}_n^{B} \cup \mathcal{E}_n^{C}$, defined by
\begin{align*}
 \mathcal{E}_1^{A} &:= \left\{ \uline{\widetilde{12}^{x}\widetilde{{12}}^{y}} \right\} \rm{,} \quad \\
 \mathcal{E}_1^{B} &:= \emptyset \rm{,} \quad \\
 \mathcal{E}_1^{C} &:= \left\{ \mathfrak{v}_1 \right\}
\end{align*}
and, for $n \geq 2$, by
\begin{align*}
 \mathcal{E}_n^{A} &:= \left\{ \uline{s_1 v_L s_2 s_3 v_r s_4} \; | \; s_k \in \{ 1,2 \} \mbox{ and } v_L v_R \in \mathcal{E}_{n-1}^{A} \cup \mathcal{E}_{n-1}^{B}  \right\} \\
 \mathcal{E}_n^{B} &:= \left\{ \uline{\widetilde{12}^{a} v_L \widetilde{12}^{b} \widetilde{12}^{c} v_R \widetilde{12}^{d}} \; | \; v_L, v_R \in \{ 1,2 \}^{2n-4} \mbox{ and } v_L v_R \notin u(\mathcal{E}_{n-2}^{A} \cup \mathcal{E}_{n-2}^{B}) \right\} \\
 \mathcal{E}_n^{C} &:= \left\{ \uline{22]^{(ad,w)} v_{L} [1122]^{(bc,z)} v_{R} [11} \; | \; v_L, v_R \in \{ 1,2 \}^{2n-4} \mbox{ and } v_L v_R \notin u(\mathcal{E}_{n-2}^{A} \cup \mathcal{E}_{n-2}^{B}) \right\}
\end{align*}
\end{itemize}
where $u : \mathcal{V}(n) \rightarrow \{1,2\}^{n}$ is the map given in Definition \ref{def:umapsfromV} and $\{v(m)\}_{m \in  \mathbb{N}}$ is the integer sequence given in Definition \ref{def:vsequence} above.
\end{theorem}

The proof depends on some combinatorial results.  First, define $\mathcal{D}(n,r)$ by:

\begin{define}  \label{def:Dnr}
Let $n,r \in \mathbb{N}$, with $0 \leq r \leq n$.  Let $k(n,r)$ be the function defined in Proposition \ref{prop:sbafacts}.

Define $\mathcal{D}(n,r) \in \mathbb{N}$ by
\begin{align}
\mathcal{D}(n,0) &= 2^{n} \label{eq:Dn0} \\  
\mathcal{D}(0,r) &= 0 \qquad \mbox{if } r \neq 0  \rm{,}  \label{eq:D0r}
\end{align}
and otherwise
\begin{align}
\mathcal{D}(n,r) &= \left(\sum_{i=0}^{k(n,r)} \begin{pmatrix} n \\ i \end{pmatrix} \right)
  + \left(\sum_{i=0}^{k(n,-r)} \begin{pmatrix} n \\ i \end{pmatrix} \right) \rm{.}
\end{align}
\end{define}

\begin{lemma}  \label{lem:Drec}
Let $n, r \in \mathbb{N}$.  Let $\mathcal{D}(n,r)$ be as in Definition \ref{def:Dnr}.  Then if $r > 0$
\begin{align}
\mathcal{D}(n+1,r) = \mathcal{D}(n,r-1) + \mathcal{D}(n,r+1) \rm{.}
\end{align}
Together with \eqref{eq:Dn0} and \eqref{eq:D0r} this recurrence relation determines $\mathcal{D}(n,r)$ exactly.

Now let $\{v(m)\}_{m \in \mathbb{N}}$ be as in Definition \ref{def:vsequence}.  Then
\begin{align}
\sum_{r=0}^{n} v(r) \mathcal{D}(n,r) &= 16^n \rm{,} 
\end{align}
and this equation also determines $\{v(m)\}_{m \in \mathbb{N}}$ exactly.
\end{lemma}
\begin{proof}
This follows immediately from some straightforward calculations.
\end{proof}

It is also useful to define sequences $\{A_n\}$ and $\{B_n\}$ as follows:

\begin{define}  \label{def:AandB}
 Define integer sequences $\{ A_n \}_{n \geq 1}$, $\{ B_n \}_{n \geq 1}$ by 
\begin{displaymath}
 A_1 = 1 \rm{,} \quad B_1 = 0 \rm{,} \quad A_2 = 16 \rm{,} \quad B_2 = 1
\end{displaymath}
and, for $n > 2$,
\begin{align}
 A_n &= 16(A_{n-1} + B_{n-1}) \\
 B_n &= 16^{n-2} - A_{n-2} - B_{n-2} \rm{,}
\end{align}
\end{define}

The next result relates these sequences to $\{v(m)\}_{m \in \mathbb{N}}$.

\begin{lemma}  \label{lem:ABresults}
 For all $n \geq 2$ we have
\begin{align}
 B_{n} &= \sum_{r=0}^{n-2} v(r) \label{eq:propAB2} \\
 A_{n} &= 16^{n} - \sum_{r=0}^{n-2} v(r) - \sum_{r=0}^{n} v(r) \label{eq:propAB3} \rm{,}
\end{align}
and hence
\begin{align}
 A_n + 2 B_n &= 16^{n} - v(n-1) - v(n) \label{eq:propAB1}
\end{align}
\end{lemma}
\begin{proof}
This statement also follows quickly from a direct calculation.
\end{proof}

\begin{proof}[Proof of Theorem \ref{thm:Vistilting}]
By induction on $n$.

The base case $n=1$ was the subject of Example \ref{example:Vagain}.

For the inductive step, suppose the theorem is true for $n=N$ and consider $n=N+1$.

By Proposition \ref{prop:Viso} it follows that $F_e^{N+1} \mathcal{V}(N+1) \cong \mathcal{V}(N)$ is tilting.  In particular, it has a standard filtration.  Then by Proposition \ref{prop:ryomhansenmartin} so does $(G_e^{N+1} \circ F_e^{N+1}) \mathcal{V}(N+1)$.  The standard multiplicities are given by

\begin{align}
\left( (G_e^{N+1} \circ F_e^{N+1}) \mathcal{V}(N+1) \; : \; \Delta_{N+1}(\lambda) \right) &= \left\{ \begin{array}{cl}
       (\mathcal{V}(N) \; : \; \Delta_{N}(-\lambda) ) & \mbox{if } -\lambda \in \Lambda_{N} \\
       0 & \mbox{otherwise} 
       \end{array} \right. \\
       &= \left\{ \begin{array}{cl}
       v(|\lambda|) & \mbox{if } -\lambda \in \Lambda_{N} \\
       0 & \mbox{otherwise}
       \end{array} \right.  \rm{.}
\end{align}

It follows that 
\begin{align}
\algdim_{\Field} (G_e \circ F_e) (\mathcal{V}(N+1)) &= \sum_{-\lambda \in \Lambda_{N}} v(|\lambda|) \algdim_{\Field} \Delta_{N+1}(\lambda) \\
 &= \sum_{r=0}^{N-1} v(r) \mathcal{D}(N+1,r) + v(N) \rm{,}
\end{align}
using the definition of $\mathcal{D}(n,r)$ and the expression for $\algdim_{\Field} \Delta_{N+1}(\lambda)$ given in Proposition \ref{prop:sbafacts}.

Now Lemma \ref{lem:Drec} implies (iii) holds for $n=N+1$, since

\begin{align}
d_{N+1} &= \algdim_{\Field} (G_e \circ F_e) (\mathcal{V}(N+1)) \\
  &= \sum_{r=0}^{N-1} v(r) \mathcal{D}(N+1,r) + v(N) \\
  &= \sum_{r=0}^{N+1} v(r) \mathcal{D}(N+1,r) - v(N) - v(N+1) \\
  &= 16^{N+1} - v(N) - v(N+1) \rm{.}
\end{align}

Now consider $\psi_{N+1} : \left(G_e^{N+1} \circ F_e^{N+1} \right) \left( \mathcal{V}(N+1) \right) \rightarrow \mathcal{V}(N+1)$.  

It follows from Proposition \ref{prop:imageunderpsi} that $\mathcal{E}_{N+1} \subset \psi_{N+1} \left( \left(G_e^{N+1} \circ F_e^{N+1} \right) \left( \mathcal{V}(N+1) \right) \right)$.  What remains is to show that this is a spanning set and that the elements are linearly independent.  Consider the second of these claims first.

By the inductive hypothesis, the elements of $\mathcal{E}_{N}$ are linearly independent.  It follows immediately that the elements of $\mathcal{E}_{N+1}^{A}$ are also independent.  The elements of $\mathcal{E}_{N+1}^{B}$ are also linearly independent, as are those of $\mathcal{E}_{N+1}^{C}$.  This follows from a direct inspection and does not require the inductive hypothesis.

Recall the map $u : \mathcal{V}(N+1) \rightarrow \{1,2\}^{N+1}$ from Definition \ref{def:umapsfromV}.

Any sequence in $u (\mathcal{E}_{N+1}^{B})$ must, by construction, have second term equal to $2$.  Any sequence in $u (\mathcal{E}_{N+1}^{C})$ must, by construction, have second term equal to $1$.  So these two sets are distinct.  Therefore so are $\mathcal{E}_{N+1}^{B}$ and $\mathcal{E}_{N+1}^{C}$.

It is also immediate from the construction of the three sets that no term that appears in $u\left( \mathcal{E}_{N+1}^{A} \right)$ can also appear in either $u\left( \mathcal{E}_{N+1}^{B} \right)$ or $u\left( \mathcal{E}_{N+1}^{C} \right)$.  So $\mathcal{E}_{N+1}^{A}$ and $\mathcal{E}_{N+1}^{B} \cup \mathcal{E}_{N+1}^{C}$ are also distinct.

This shows that the elements of $\mathcal{E}_{N+1} := \mathcal{E}_{N+1}^{A} \cup \mathcal{E}_{N+1}^{B} \cup \mathcal{E}_{N+1}^{C}$ are linearly independent.

It remains to show that $\mathcal{E}_{N+1}$ is a spanning set of $\psi_{N+1} \left( \left(G_e^{N+1} \circ F_e^{N+1} \right) \left( \mathcal{V}(N+1) \right) \right)$.  This can be shown by demonstrating that the dimension of the subspace spanned by $\mathcal{E}_{N+1}$ is equal to the dimension of the whole space $\psi_{N+1} \left( \left(G_e^{N+1} \circ F_e^{N+1} \right) \left( \mathcal{V}(N+1) \right) \right)$.  Showing this will also prove that part (iii) of the Theorem holds for $n=N+1$.

Observe first that 
\begin{displaymath}
\algdim_{\Field} \psi_{N+1} \left( (G_e \circ F_e) \mathcal{V}(N+1) \right) \leq \algdim_{\Field} (G_e \circ F_e) \mathcal{V}(N+1) \rm{.}
\end{displaymath}

Observe also that, from the definitions of the sets $\mathcal{E}_{n}^{A}$, $\mathcal{E}_{n}^{B}$ and $\mathcal{E}_{n}^{C}$,
\begin{align}
|\mathcal{E}_{n}^{A}| &= A_n \\
|\mathcal{E}_{n}^{B}| &= B_n \\
|\mathcal{E}_{n}^{C}| &= B_n \rm{,}
\end{align}
where $\{A_n\}$ and $\{B_n\}$ are the sequences defined recursively in Definition \ref{def:AandB}.

Therefore, by Lemma \ref{lem:ABresults} and the linear independence of the elements of $\mathcal{E}_{N+1}$, it follows that
\begin{align}
|\mathcal{E}_{N+1}| &= |\mathcal{E}_{N+1}^{A}| + |\mathcal{E}_{N+1}^{B}| + |\mathcal{E}_{N+1}^{C}| \\
   &= A_{N+1} + 2B_{N+1} \\
   &= 16^{N+1} - v(N) - v(N+1) \\
   &= d_{N+1} \rm{.}
\end{align}

On the other hand, by Lemma \ref{lem:Drec},

\begin{align*}
\algdim_{\Field} (G_e \circ F_e) (\mathcal{V}(N+1)) &= \sum_{\lambda \in \Lambda_{N+1} \setminus \{-N,N-1\}} v(|\lambda|) \algdim_{\Field} \Delta_{N+1}(\lambda) \\
&= \algdim_{\Field} \Delta_{N+1}(0) v(0) - v(N) \algdim_{\Field} \Delta_{N+1}(N) - v(N+1) \algdim_{\Field} \Delta_{N+1}(-(N+1))\\
&\qquad + \sum_{r=1}^{N+1} \left( \algdim_{\Field} \Delta_{N+1}(r) + \algdim_{\Field} \Delta_{N+1}(-r)  \right) v(r)  \\
&= \sum_{r=0}^{N+1} \mathcal{D}(N+1,r) v(r) - v(N) - v(N+1) \\
&= 16^{N+1} - v(N) - v(N+1) \\
&= d_{N+1} \rm{.}
\end{align*}

Thus we have shown both (iii) and (iv).

Finally, for part (ii), recall the localisation functor $F_f$.  By Proposition \ref{prop:Viso}, $F_e \mathcal{V}(N+1) \cong F_f \mathcal{V}(N+1) \cong \mathcal{V}(N)$. The standard multiplicity
\begin{displaymath}
\left[ \mathcal{V}(N+1) : \Delta_{N+1}(N) \right] = \left[ \mathcal{V}(N+1) : \Delta_{N+1}(-N) \right] = v(N) \rm{.}
\end{displaymath}
To find $\left[ \mathcal{V}(N+1) : \Delta_{N+1}(-(N+1)) \right]$, note that
\begin{align*}
\algdim_{\Field} \mathcal{V}(N+1) &= 16^{N+1} \\
 &= \sum_{r=0}^{N} v(r) \mathcal{D}(N+1,r) + \left[ \mathcal{V}(N+1) : \Delta(N-1) \right] \rm{.}
\end{align*}
Therefore $\left[ \mathcal{V}(N+1) : \Delta(N-1) \right] = v(N+1)$ by Lemma \ref{lem:Drec}.  This proves (ii) for $n=N+1$ and completes the proof of the theorem.
\end{proof}

The following consequence of Theorem \ref{thm:Vistilting} is immediate.
\begin{corollary}
Let $n \in \mathbb{N}$. Let $\uline{\Pi} = (\Parameters) \in (\Field^{\times})^{6}$.

Then each of the permutation modules $M_n(r)$ given in Definition \ref{def:permutationmodule} is a tilting module for $\sba(\uline{\Pi})$.
\end{corollary}
\begin{proof}
Every summand of a tilting module is itself tilting.  By Theorem \ref{thm:Vistilting} $\themodule$ is tilting.  By Proposition \ref{prop:permutationmodules} $\themodule$ is a direct sum of permutation modules, and every permutation module $M_n(r)$ appears as a direct summand.  This proves the claimed result.
\end{proof}

\begin{lemma}  \label{lem:FonMnr}
\begin{displaymath}
F_e^{n} M_n(r) \cong \left\{ \begin{array}{cl}
                     M_{n-1}(r) & \mbox{if } r \neq n \\
                     0 & \mbox{if } r = n \end{array}  \right. \rm{.}
\end{displaymath}
\begin{displaymath}
F_f^{n} M_n(r) \cong \left\{ \begin{array}{cl}
                     M_{n-1}(r) & \mbox{if } r \neq n \\
                     0 & \mbox{if } r = n \end{array}  \right. \rm{.}
\end{displaymath}
\end{lemma}
\begin{proof}
It is immediate from Proposition \ref{prop:Viso} that $F_e^{n} M_n(r) \subseteq M_{n-1}(r)$ and that $F_f^{n} M_n(r) \subseteq M_{n-1}(r)$.  Equality follows from Theorem \ref{thm:Vistilting} and Proposition \ref{prop:Fexact}, specifically the fact that $F_e^{n}$ and $F_f^{n}$ are exact.
\end{proof}

\begin{lemma}  \label{lem:smallestpermutationistilting}
\begin{displaymath}
M_n(4n) = L_n(-n) = \Delta_n(-n) = \nabla_n(-n) = T_n(-n)
\end{displaymath}
\end{lemma}
\begin{proof}
As $M_n(4n)$ is the module generated by $\uline{11\ldots1}$ and this vector is killed by all elements of $\mathcal{G}_n$ it follows that $M_n(n)$ is one-dimensional.  Therefore $M_n(4n)$ is simple.  Since $M_n(4n)$ is killed by both $F_e$ and $F_f$, it then follows by Proposition \ref{prop:Fexact} that $M_n(4n) \cong L_n(-n)$.  Since $M_n(4n)$ is tilting (by Lemma \ref{lem:FonMnr}) it follows also that $M_n(n) \cong \Delta_n(-n) \cong \nabla_n(-n) \cong T_n(-n)$.
\end{proof}

\begin{theorem} \label{thm:fulltilt}
Let $n \in \mathbb{N}$.  Let $\uline{\Pi} \in (\Field^{\times})^{6}$.

Then the module $\tensor[_{\sba(\uline{\Pi})}]{\mathcal{V}(n)}{}$ is full-tilting.
\end{theorem}
\begin{proof}
By induction on $n$.

By Theorem \ref{thm:Vistilting} $\mathcal{V}(n)$ is tilting. Therefore it is the sum of some indecomposable tilting modules:
\begin{align}
\mathcal{V}(n) = \bigoplus_{\lambda \in \Lambda_n} s_{\lambda}  T_n(\lambda) \rm{,}
\end{align}
for $s_{\lambda} \in \mathbb{N}$.

Observe that $\mathcal{V}(n)$ is full-tilting exactly if $s_{\lambda} > 0$ for every $\lambda \in \Lambda_n$.  Observe also that it follows from Proposition \ref{prop:Fexact} that for any $\lambda \in \Lambda_n$

\begin{align*}
F_{e} T_{n} (\lambda) &= \left\{ \begin{array}{cl}
   T_{n-1} (-\lambda )  & \mbox{if } -\lambda \in \Lambda_{n-1} \\
   0  & \mbox{otherwise}
 \end{array} \right.  \rm{,}
\end{align*}
and 
\begin{align*}
F_{f} T_{n} (\lambda) &= \left\{ \begin{array}{cl}
   T_{n-1} (\lambda )  & \mbox{if } \lambda \in \Lambda_{n-1} \\
   0  & \mbox{otherwise}
 \end{array} \right.  \rm{.}
\end{align*}

It follows therefore that
\begin{align}
F_e (\mathcal{V}(n)) &= F_e \left( \bigoplus_{\lambda \in \Lambda_n} s_{\lambda}  T_n(\lambda) \right) \\
  &= \bigoplus_{\lambda \in \Lambda_n} s_{\lambda}  F_e (T_n(\lambda)) \\
  &= \bigoplus_{\lambda \in \Lambda_{n-1}} s_{-\lambda} T_{n-1}(-\lambda) \rm{,} \\
F_f (\mathcal{V}(n)) &= F_f \left( \bigoplus_{\lambda \in \Lambda_n} s_{\lambda}  T_n(\lambda) \right) \\
 &= \bigoplus_{\lambda \in \Lambda_n} s_{\lambda}  F_f (T_n(\lambda)) \\
 &= \bigoplus_{\lambda \in \Lambda_{n-1}} s_{\lambda} T_{n-1}(\lambda) \rm{.}
\end{align}

Since $\mathcal{V}(n-1)$ is full-tilting, it follows that $s_{\lambda} > 0$ for $\lambda \in \Lambda_{-n} \setminus \{-n\}$.

But by Lemma \ref{lem:smallestpermutationistilting} $M_n(2n) \cong T_n(-n)$.  Therefore $s_{-n} \geq 1$.

Hence $\mathcal{V}(n)$ is full-tilting for all $n \in \mathbb{N}$, as claimed.
\end{proof}

\section{Discussion}

Let $\Parameters \in \Field^{\times}$.  Let $\uline{\Pi} = (\Parameters)$.  Let $n \in \mathbb{N}$.

\subsection{Further Multiplicities}

It is of some interest to compare the standard filtration of $\themodule$ with the direct sum decomposition into ``permutation modules''.  Several problems suggest themselves.

Theorem \ref{thm:Vistilting} gives an expression for $( \mathcal{V}(n) \; : \; \Delta_n(\lambda))$.  In fact, the sequence $\{v(r)\}_{r \in \mathbb{N}}$ given in Definition \ref{def:vsequence} can be described in terms of Chebyshev polynomials:
\begin{equation}
v(r) = \left\{ \begin{array}{cl}
                 1 & \mbox{if } r = 0 \\
                 14 U_{n-1}(8) & \mbox{if } r > 0
               \end{array}  \right.  \rm{,} 
\end{equation}
where $U_n(x)$ is a Chebyshev polynomial of the second kind (see for example Definition 1.2 of \cite{masonhandscomb}), given explicitly by
\begin{equation}
 U_n(x) = \frac{\left( x + \sqrt{x^2-1} \right)^{n+1} -  \left( x - \sqrt{x^2-1}\right)^{n+1}}{2\sqrt{x^2-1}} \rm{.}
\end{equation}

Is it possible to obtain a similarly explicit formula for $(M_n(r) \; : \: \Delta_n(\lambda))$?  Note that $M_n(r) \cong M_n(-r)$ for any $-2n \leq r \leq 2n$.  So $(M_n(r) \; : \: \Delta_n(\lambda)) = (M_n(-r) \; : \: \Delta_n(\lambda))$.  Since $F_e \mathcal{V}(n) \cong \mathcal{V}(n-1) \cong F_f \mathcal{V}(n)$, it must also be the case that $(M_n(r) \; : \: \Delta_n(\lambda)) = (M_n(r) \; : \: \Delta_n(-\lambda))$ whenever $\lambda, -\lambda \in \Lambda_n$, and that $(M_n(r) \; : \: \Delta_n(\lambda)) = (M_{n-1}(r) \; : \: \Delta_n(\lambda))$ whenever $-2(n-1) \leq r \leq 2(n-1)$.

The table below gives $(M_n(r) \; : \: \Delta_n(\lambda))$ for some small values of $n \in \mathbb{N}$.

\begin{displaymath}
\begin{array}{cc|ccccc}
r & |\lambda| & 0 & 1 & 2 & 3 & 4 \\
\hline
0 && 1 & 4 & 58 & 780 & 10906 \\
1 &&   & 4 & 48 & 676 & 9760 \\
2 &&   & 1 & 26 & 438 & 6966 \\
3 &&   &   &  8 & 204 & 3912 \\
4 &&   &   &  1 &  64 & 1686 \\
5 &&   &   &    &  12 & 536 \\
6 &&   &   &    &   1 & 118 \\
7 &&   &   &    &     & 16 \\ 
8 &&   &   &    &     & 1
\end{array}
\end{displaymath}

Compare with the similar table in Section 4.1 of \cite{martinryomhansen2004}.  That table gives similar multiplicities for ``permutation modules'' for the \emph{ordinary} blob algebra.  The ordinary blob algebras $\{b_n\}_{n \in \mathbb{N}}$ are a tower of algebras in the sense of \cite{coxmartinparkerxi2006}.  In particular for any $n \geq 1$ the algebra $b_{n-1}$ is isomorphic to a subalgebra of $b_n$ and therefore there is a restriction functor
\begin{displaymath}
\res : b_{n}-\amod \rightarrow b_{n-1}-\amod \rm{.}
\end{displaymath}

Using this functor the authors of \cite{martinryomhansen2004} give a recursion formula for the multiplicity of any standard module for the blob algebra as a factor of any permutation module.  The symplectic blob algebras $\{\sba\}_{n \in \mathbb{N}}$ are not a tower of algebras in this sense and at present we do not have even a candidate recursion formula for the multiplicities $(M_n(r) \; : \: \Delta_n(\lambda))$.

\subsection{Ringel Duals of $\sba$}

The algebra $\sba = \sba(\uline{\Pi})$ is quasihereditary, and the module $\themodule$ is a full tilting $\sba$-module.  Therefore $E_n := \End_{\sba}(\mathcal{V}(n))$ is also a quasihereditary algebra.  It is called the \emph{Ringel dual} of $\sba(\uline{\Pi})$.  

Note that the submodule $\mathcal{I}(n) \subset \mathcal{V}(n)$ defined by
\begin{displaymath}
\mathcal{I}(n) := \bigoplus_{0 \leq s \leq n} M_n(2s) \rm{,}
\end{displaymath}
is also a full-tilting module.  As such the algebra $\End_{\sba}(\mathcal{I}(n))$ is Morita equivalent to $E_n$.

One of the goals of the ``virtual algebraic Lie theory'' programme described in \cite{martinryomhansen2004} is to find an explicit presentation of such Ringel duals, ideally in a familiar (Lie-theoretic) setting.  An additional challenge to this programme in the case of the symplectic blob algebra is that the non-generic representation theory of $\sba$ is not at present fully understood.

\subsection*{Acknowledgements}

While writing this paper the author received funding from an EPSRC training grant. He would like to thank his supervisor, Paul Martin, for introducing him to this field and suggesting this problem, and for many useful discussions and suggestions. The author would also like to thank Elizabeth Banjo and Alison Parker, both of whom read and commented on previous drafts of this paper.

\bibliography{tilting}

\providecommand{\bysame}{\leavevmode\hbox to3em{\hrulefill}\thinspace}
\providecommand{\MR}{\relax\ifhmode\unskip\space\fi MR }
\providecommand{\MRhref}[2]{%
  \href{http://www.ams.org/mathscinet-getitem?mr=#1}{#2}
}
\providecommand{\href}[2]{#2}
\begin{thebibliography}{10}

\bibitem{assemsimsonskowronski}
I.~Assem, D.~Simson, and A.~Skowronski, \emph{Elements of the representation
  theory of associative algebras}, vol.~1, Cambridge University Press, 2006.

\bibitem{coxmartinparkerxi2006}
A.~G. Cox, P.~P. Martin, A.~Parker, and C.~C. Xi, \emph{Representation theory
  of towers of recollement: theory, notes and examples}, Journal of Algebra
  \textbf{302} (2006), 340--360, arXiv:math/0411395v1.

\bibitem{degiernichols2009}
J.~de~Gier and A.~Nichols, \emph{The two boundary {T}emperley-{L}ieb algebra},
  Journal Of Algebra \textbf{321} (2009), 1132--1167, arXiv:math/0703338.

\bibitem{dengduparshallwang}
B.~Deng, J.~Du, B.~Parshall, and J.~Wang, \emph{Finite dimensional algebras and
  quantum groups}, American Mathematical Society / Providence, Rhode Island,
  2008.

\bibitem{donkin}
S.~Donkin, \emph{The q-{S}chur algebra}, Cambridge University Press, 1998.

\bibitem{greenmartinparker2007}
R.~M. Green, P.~P. Martin, and A.~E. Parker, \emph{Towers of recollement and
  bases for diagram algebras: Planar diagrams and a little beyond}, Journal Of
  Algebra \textbf{316} (2007), 392--452, arXiv:math/0610971v1.

\bibitem{greenmartinparker2008}
\bysame, \emph{On the non-generic representation theory of the symplectic blob
  algebra}, 2008, (preprint), arXiv:math/0610971v1.

\bibitem{maclane}
S.~Mac Lane, \emph{Categories for the working mathematician}, 2nd ed.,
  Springer, 1997.

\bibitem{martinryomhansen2004}
P.~P. Martin and S.~Ryom-Hansen, \emph{Virtual algebraic {L}ie theory: Tilting
  modules and {R}ingel duals for blob algebras}, Proceedings of the LMS
  \textbf{89} (2004), 655--675, arXiv:math/0210063.

\bibitem{masonhandscomb}
J.~C. Mason and D.~C. Handscomb, \emph{Chebyshev polynomials}, Chapman and
  Hall, 2002.

\bibitem{nichols2006}
A.~Nichols, \emph{The {T}emperley-{L}ieb algebra and its generalisations in the
  {P}otts and {XXZ} models}, Journal of Statistical Mechanics: Theory and
  Experiment (2006), arXiv:hep-th/0509069v2.

\bibitem{reeves2011}
Andrew Reeves, \emph{A tensor space representation of the symplectic blob
  algebra}, 2011, (preprint), arXiv:1111.0145v1.

\end{thebibliography}
\bibliographystyle{amsplain}

\end{document}